\documentclass[12pt]{amsart}
\usepackage{fullpage}
\usepackage[utf8]{inputenc}
\usepackage[english]{babel}
\usepackage{amsthm}
\usepackage{amsmath}
\usepackage{amssymb}
\usepackage{multirow}
\usepackage{appendix}
\usepackage{mathtools}
\usepackage[bottom]{footmisc}
\usepackage{url}

\newtheorem{theorem}{Theorem}[section]
\newtheorem{corollary}{Corollary}[theorem]
\newtheorem{lemma}[theorem]{Lemma}
\newtheorem{proposition}[theorem]{Proposition}
\newtheorem*{remark}{Remark}
\newtheorem{definition}{Definition}[section]
\newtheorem{example}{Example}[section]
\newtheorem{fact}{Fact}[section]

\newcommand{\QQ}{\mathbb{Q}}
\newcommand{\RR}{\mathbb{R}}
\newcommand{\ZZ}{\mathbb{Z}}
\newcommand{\CC}{\mathbb{C}}

\newcommand{\HH}{\mathbb{H}}
\DeclareMathOperator{\Nm}{Nm}
\DeclareMathOperator{\Cor}{Cor}

\title{K3 surfaces associated to Abelian Fourfolds of Mumford's Type}
\author{Yuwei Zhu}

\begin{document}
\maketitle

\begin{abstract}

Mumford constructed a family of abelian fourfolds with special stucture not characterized by endomorphism ring. Galluzzi showed that the weight 2 Hodge structure of such a variety decomposes into Hodge substructures via the action of Mumford-Tate group, one of which is of K3 type with Hodge number $(1,7,1)$. We will compute the intersection form of such a Hodge structure and provide a canonical $\ZZ$-Hodge structure on this subspace. Furthermore, we shall show two applications of this invariant. Firstly, using the above formula we shall show that every K3 surface obtained in this way will admit an elliptic fibration. Comparing them with the list by Shimada, we will show that the Mordell-Weil group of all elliptic fibrations have either 2-torsion or trivial torsion group, and the conditions of the occurance of 2-torsions can be specified. Secondly, we shall use it to determine the quaternion that gives a special CM abelian fourfold that is derived equivalent to the jacobian of $y^2 = x^9 -1$.
\end{abstract}

\section{Introduction}

In \cite{Mum69} Mumford constructed a family of abelian fourfolds with extra Hodge cycles that are not generated by the intersection of divisors. Furthermore, the Mumford-Tate group of a generic member is not $Sp_8({\QQ})$ while the endomorphism ring of such a member is $\ZZ$. Since then, a more explicit characterization has been made with regard to special cases (cf. \cite{Noo00} for the discussion on $l$-adic representations for the CM case and \cite{Gal99} for the explicit description of CM cases satisfying some arithmetic conditions), but a universal formula for computing meaningful invariants for general cases remains unknown.

It is a general fact that a quaternion $B$ defined over a number field $K$ can be uniquely determined by its places of ramification $Ram(B)$, which is the set of places such that locally the quaternion is not isomorphic to $2 \times 2$ matrix. The product of all finite places in $Ram(B)$ is called the discriminant of $B$, denoted $disc(B)$. It is also a general fact that for a quaternion there exists at least one integral model of $B$, called a maximal order $O$ of $B$.

The construction given by Mumford goes as follows: Suppose $B$ a quaternion defined over a totally real cubic field $K$ satisfies the following two conditions: 1) $B \otimes_{\QQ} \RR \cong \HH \oplus \HH \oplus M_2(\RR)$ and 2) $\Cor_{K/\QQ}B$ is trivial. Then the norm-1 subgroup $G$ of $B$, regarded as a $\QQ$-group, can be realized as a Mumford-Tate group of an abelian fourfold. Classical result in Lie theory states that the Lie algebra of $G$, when we view $G$ as a $K$-group is a vector space of $K$-dimension $3$, which coincides with the trace zero space $B^0 \subset B$. $B^0$ is therefore naturally endowed with a $K$-quadratic form invariant under the action of $G$, namely the Killing form $Q_B$.

This paper is organized as follows: In section 2 of this paper we shall recall some basic facts about quaternions, Mumford's construction and Galluzzi's result over $\CC$. In section 3, we will define an action called corestriction (of quadratic spaces) $$\Cor_{K/\QQ}: \{K \text{-quadratic spaces}\} \rightarrow \{\QQ \text{-quadratic spaces}\} $$ which is compatible with the corestriction of central simple algebras given by Mumford. Using this action we will give an explicit description of $V_{K3}$ over $\QQ$ and write down a representative of the intersection forms on $V_{K3}$, which is obtained by performing corestriction on the quadratic space $(B^0, Q_B)$, and it gives exactly the $\QQ$-space Galluzzi calculated over $\CC$. We will also calculate the discriminant of this representative form, and construct a canoncial lattice from a maximal order $O$:

\begin{theorem}
$V_{K3}$ along with the canonical intersection form is given by the corestriction of the Lie algebra of $G$ and the Killing form: $$(V_{K3}, Q_{\text{can}}) = \Cor_{K/\QQ}(B^0, Q_B)$$
Furthermore, if we are given a maximal order $O$ of $B$, there is a canonical lattice $\Lambda_{can}$ obtained by intersecting $B^0$ with $O$. The discriminant of $\Lambda_{can}$ with respect to the canonical intersection form is $$disc(\Lambda_{can}) = -2^3disc(K/\QQ)^3 Norm_{K/\QQ}(disc(B))^2$$
\end{theorem}

We will apply this invariant in two ways in section 4. With some input from the theory of elliptic fibrations on K3 surface and basic algebraic number theory, we can use the above formula to prove the following theorem:
\begin{theorem}
Any K3 surface with its transcendental Hodge structure arising from Mumford's construction admits elliptic fibrations with section, and the torsion parts of the Mordell-Weil group of such fibrations are always trivial if the intersection form is $Q_{can}$, except when the canonical intersection form can be realized as the discriminant form of a rank $3$ lattice. In that case, the Mordell-Weil torsion group can be $\ZZ/2\ZZ$.
\end{theorem}
Next we are going to use a twisted form of this invariant to pinpoint the quaternion that gives the special Mumford fourfold $A_M$ that is proven to be derived equivalent to Shioda's fourfold $A_S$ in \cite{myself}:

\begin{theorem}
$A_M$ lies on the Shimura curve given by the quaternion $B$ defined over the totally real cubic subfield $\QQ(\zeta_9)$ that ramifies only at two infinite places.
\end{theorem}

\renewcommand{\abstractname}{Acknowledgments}

\begin{abstract}
The author would like to thank Prof. Kiran Kedlaya for pointing to the subject of Mumford's construction, Prof. John Voight for the many helpful and enlightening discussions on the arithmetic properties of quaternion algebras, Prof. Noam Elkies for pointing to a neglected case in the K3 fibrations, and Prof. Brendan Hassett for his valuable suggestions and patient guidance throughout the years. The author is supported by the NSF grants DMS-1551514 and DMS-1701659.
\end{abstract}

\section{Background}

\subsection{Arithmetics of quaternions over number fields}

The main reference for this section is John Voight's book on quaternion algebras \cite{Voi17}.

\begin{definition} (Definition 2.2.1 in \cite{Voi17})
An algebra $B$ over a field $K$ is called a \textbf{quaternion algebra} if there is a basis $1, \alpha, \beta, \alpha\beta$ for $B$ as a vector space over $K$ such that $$\alpha^2 = a, \beta^2 = b, \alpha\beta = -\beta\alpha$$ for some $a, b \in K^{\times}$. We denote such a quaternion to be $\left (\frac{a,b}{K} \right )$.
\end{definition}

It is a standard fact that quaternions classify $K$-forms of $PSL_2$, which is also classified by $K$-forms of $\mathbb{P}^1$, namely the Brauer-Severi varieties of dimension 1 (cf. Chapter III Example 1.4 and Exercise 3 in \cite{Ser13}). Since being a class of such $K$-forms implies that there exists some field $F/K$ such that $B\otimes F = M_2(F)$, we can define the \textbf{reduced trace of an element} $\gamma$ to be the following:  $trd(\gamma) := tr(\gamma \otimes 1)$, where $tr$ is the usual trace of a matrix and $\gamma \otimes 1 \in M_2(F)$. The space of all reduced trace zero elements is denoted $B^0$. Similarly we can define the analogue of determinant on $B$, which is called \textbf{reduced norm}, denoted $nrd$.

The connection between a quaternion algebra up to isomorphism and a Brauer-Severi variety of dimension 1 up to birational equivalence is made explicit in the following definition:

\begin{definition}
The quadratic form $ax^2 + by^2 = z^2$ where $a, b$ coming from $B = \left (\frac{a,b}{K} \right )$ is called the \textbf{conic associated to the quaternion $\left (\frac{a,b}{K} \right )$}. The places (finite or infinite) where the conic does not have a solution are called places of ramification. The set of all such places is denoted $Ram (B)$.

(Definition 14.5.4 in \cite{Voi17}) The product of all finite places in $Ram (B)$ is called the discriminant of $B$, denoted $\mathfrak{D}_B$ or $\mathfrak{D}$ if $B$ is understood. It is an ideal in $\mathcal{O}_K$
\end{definition}

Combining these facts with Hasse principle we have the following theorem:

\begin{theorem} (Main Theorem 14.6.1 in \cite{Voi17}) Let $K$ be a global field. Then the map $B \mapsto Ram (B)$ gives a bijection between quaternions over $K$ up to isomorphisms and finite, even cardinality subsets of noncomplex places of $K$.

\end{theorem}

\subsection{Mumford's construction}

In this section we introduce the main ingredients of Mumford's construction in \cite{Mum69}. We denote $K$ to be the totally real cubic extension. $K^{(1)}$, $K^{(2)}$, $K^{(3)}$ denotes the three embeddings of $\QQ$-algebras $\sigma_i: K \hookrightarrow \overline{\QQ}$.

Given a quaternion $B$ we define the action of taking \textbf{corestriction}. Let $B^{(i)} = B \otimes_K K^{(i)}$ be the quaternion algebra over $\overline{\QQ}$. Tensoring them together we get a central simple algebra $D = B^{(1)} \otimes B^{(2)} \otimes B^{(3)}$. The Galois group $Gal(\overline{\QQ}/\QQ)$ has a natural semilinear action on $D$ given by the following: if $\tau: \overline{\QQ} \rightarrow \overline{\QQ}$ is an element in $Gal(\overline{\QQ}/\QQ)$, then $\tau \circ \sigma_i = \sigma_{\pi(i)}$ for some permutation $\pi$. Similarly $\tau$ acts on $\overline{\QQ}$, which then gives a semilinear isomorphism $B^{(i)} \xrightarrow{\sim} B^{\pi(i)}$, thus a semilinear automorphism on $D$. We define the elements fixed by such Galois actions in $D$ to be $\Cor_{K/\mathbb{Q}}(B)$. It is now a central simple algebra defined over $\mathbb{Q}$.

Imitating the norm map and the trace map in number theory, there are immediately two flavors of invariant elements in $\Cor_{K/\mathbb{Q}}(B)$ that we could obtain from $B$: the image of the norm map 
$$\begin{array}{cccc}
\Nm: & B & \rightarrow & \Cor_{K/\QQ}(B) \\
 & \gamma & \mapsto & \gamma^{(1)} \otimes \gamma^{(2)} \otimes \gamma^{(3)} \\
\end{array}$$
and the sum of Galois orbits:
$$\phi: \gamma \mapsto \gamma^{(1)} \otimes 1\otimes 1 + 1\otimes \gamma ^{(2)} \otimes 1 + 1\otimes 1\otimes \gamma ^{(3)}$$
Mumford's construction of abelian fourfold goes as follows: we start with a quaternion $B$ satisfying these two conditions:
$$ \Cor_{K/\mathbb{Q}}(B) \cong M_8(\mathbb{Q})$$
$$ B \otimes_{\mathbb{Q}} \mathbb{R} \cong \mathbb{H} \oplus \mathbb{H} \oplus M_2({\mathbb{R}})$$
Now we let $G$ denote the reduced norm one elements in $B$ i.e. 
$$G := \left\{ \gamma \in B | nrd(\gamma) = 1 \right\}$$
We view $G$ as a $\QQ$ group. $G \otimes_{\mathbb{Q}} \mathbb{R} \cong SU(2) \times SU(2) \times SL_2(\mathbb{R})$. The map $Nm$ maps it to a subgroup in $GL_8(\mathbb{Q})$, which yields to an algebraic representation of $G$ in $V = \mathbb{Q}^8$ defined over $\mathbb{Q}$. It is a $\mathbb{Q}$-form of $\mathbb{R}$-representation:
$$SU(2) \times SU(2) \times SL_2(\mathbb{R}) \rightarrow SO(4) \times SL_2(\mathbb{R} ) \subset Sp_8(\RR)$$
Therefore $G_{\RR}$ fixes a unique symplectic form that can be explicitely written down over $\RR$. To see that the symplectic form is unique, one could calculate the $H^2 \otimes \CC = \wedge^2 V_{\CC}$ as a representation of $G_{\CC}$, and see that indeed there is only one copy of trivial representation.

The embedding of Deligne's torus $\mathbb{S} := Res_{\mathbb{C}/\mathbb{R}}(\mathbb{C}^*)$ is given by the map $$h: e^{i\theta} \rightarrow  I_4 \otimes \left ( \begin{array}{cc}
cos\theta & sin\theta \\
-sin\theta & cos\theta   \end{array} \right)$$

Now, take an arbitrary lattice $\Lambda$ in $V$ and let the monodromy group $\Gamma$ be given by $\Gamma := GL(\Lambda, \mathbb{Q}) \cap G$. Then $(V, G, \Nm, h, \Lambda, \Gamma)$ gives the necessary data for a polarized abelian fourfold $A$.

\subsection{Galluzzi's Analysis on the representation $\Nm$}

By analyzing the behavior of the roots, in \cite{Gal00} Galluzzi proves the following results.

\begin{proposition}
(Proposition 3.1 in \cite{Gal00}) The representation of $G_{\mathbb{C}}$ induced by $\Nm$ on $V_{\mathbb{C}} = \mathbb{C}^8$ is given by $$V_{\mathbb{C}} \cong W_{1,0,0} \boxtimes W_{0,1,0} \boxtimes W_{0,0,1} $$
\end{proposition}

\begin{corollary}
There is an irreducible subrepresentation of $G$ in $V \otimes V$ such that, when we base change to $\mathbb{C}$, it is isomorphic to $W_{2,0,0} \oplus W_{0,2,0} \oplus W_{0,0,2}$. Furthermore, its $\mathbb{Q}$-form is a non-split representation of $G$. Since the action of $\mathbb{S}$ is given on $V$, it can be shown that this representation is a Hodge-substructure of Hodge number $(1, 7,1)$.

\end{corollary}

This 9-dimensional subrepresentation will be of central interest in this paper. We will denote the $\mathbb{Q}$-form of it to be $V_{K3}$. In the following sections, we are going to calculate the intersection form induced by the symplectic form of $V$ on $V_{K3}$, construct a canonical lattice from $B$, which will give rise to a transcendental lattice of some K3 surface of Picard rank 13.

\section{The corestriction of Lie Algebra, and a canonical integral lattice}

In this section we give explicit description of the space $V_{K3}$ and an explicit formula of the intersection form based on the generators of $B$. Results in 3.1 hold true in general even if we drop the assumption that $\Cor_{K/\QQ}(B)$ is trivial, but for our convenience we shall assume $B$ does not ramify at primes lying over $2$.

\subsection{The corestriction of quadratic spaces}

The result of this section generalizes to any quaternion algebra defined over a totally real field of any degree, as long as it satisfies the condition that there is only one infinite place at which the quaternion algebra splits.

\begin{lemma}
(Proposition 4.3 in \cite{Gal00}) The Lie algebra of $G$, when viewed as a $K$ group, is the space of trace zero elements in $B$ (i.e. $B^0$), spanned by $\alpha, \beta, \alpha\beta$.
\end{lemma} 

\begin{proof} It is a general fact that the Lie algebra of $SL_2(F)$ is the trace zero $2\times 2$ matrices for any number field $F$. Take $L$ to be a splitting field of $B$ (i.e. $B \otimes_K F \cong M_2(F)$). By the way we defined reduced norm and reduced trace, we have that for any $\gamma \in B$, $trd(\gamma) = tr(\gamma)$, and $nrd(\gamma) = det(\gamma)$. Since $G \otimes F \cong SL_2(F)$, we see that $Lie(G) \otimes F = B^0 \otimes F$. Noticing that being traceless is an invariant property up to field extension, we conclude that $Lie(G) = B^0$.

\end{proof}

\begin{remark}
The Killing form on $B^0$ recovers the conic associated to the quaternion algebra: with respect to the basis $\{ \frac{1}{2}\alpha, \frac{1}{2}\beta, \frac{1}{2}\alpha \beta \}$, the Killing form is given by
$$ Q_B = \left ( \begin{array}{ccc}
2a &  & \\
 & 2b & \\
 & & -2ab \\
\end{array}
 \right ) $$
We can renormalize it such that this matrix form is with respect to the basis $\{\alpha,\beta,\alpha \beta \}$
\end{remark}

We now introduce a construction called corestriction of quadratic space. Let $K$ be an arbitrary totally real field of degree $d$, $(V,Q)$ a vector space of dimension $n$ with a quadratic form $Q$. Let $\sigma_1, ..., \sigma_d$ be the $d$ distinct embeddings of $K$ into $\RR$. First we consider the space $$V \otimes_{\QQ} \RR \cong (K^{\sigma_1})^n \oplus (K^{\sigma_2})^n \oplus ... (K^{\sigma_d})^n$$ where $K^{\sigma_i}$ denotes $\RR$ with $K$ embedded into it via $\sigma_i$. Under this isomorphism, a vector $kv = v \otimes k$ in $V$ will be mapped to $ (\sigma_1(k)v^{(1)}, ..., \sigma_d(k)v^{(d)})$.
This space is endowed with a semilinear automorphism by the Galois group $G = Gal(K/\QQ)$: if $\sigma_j = \sigma_i \circ g$ for a $g \in G$, then $g$ takes $\sigma_i(k)\otimes v^{(i)}$ to $\sigma_j(k) \otimes v^{(j)}$.

\begin{definition}
The corestriction of $V$ from $K$ to $\QQ$ is defined to be the $\QQ$-subspace in $V \otimes_{\QQ} \RR$ invariant under the semilinear automorphisms by $G$. i.e.
$$\Cor_{K/\QQ} (V) = (V \otimes_{\QQ} \RR)^G $$
Similarly, use $Q^{\sigma_i}$ to denote the $\RR$-quadratic form $Q^{\sigma_i}( r_1 \otimes \sigma_i(x), r_2 \otimes \sigma_i(y)) := r_1 \cdot r_2 \sigma_i (Q(x,y))$. The corestriciton of $Q$, denoted $Q_0$ is the quadratic form on $(K^{\sigma_1})^n \oplus (K^{\sigma_2})^n \oplus ... (K^{\sigma_d})^n$ given by $Q^{\sigma_1} \oplus Q^{\sigma_2} \oplus ... Q^{\sigma_d} $. The $\QQ$ quadratic space we obtain in this way is denoted $\Cor_{K/\QQ}(V, Q)$, called \textbf{corestriction of the quadratic space $(V,Q)$}. Likewise, if $\Lambda$ is a projective $\mathcal{O}_K$ lattice in $(V,Q)$, we can define the corestriction of $\Lambda$ and denote it as $\Cor_{K/\QQ}(\Lambda, Q)$.
\end{definition}

\begin{lemma}
$Q_0$ takes coefficients in $\QQ$ on $\Cor_{K/\QQ}(V)$. If $Q$ is an $\mathcal{O}_K$ form, then $Q_0$ is integral. Furthermore, we have $$det (Q_0) = disc(K/\QQ)^n \cdot Norm_{K/\QQ}(det(Q))$$
\end{lemma}

\begin{proof}
The rationality of $Q_0$ and $\Cor_{K/\QQ}(V)$ follows from the observation that they are invariant under the semilinear Galois action.

To prove the identity between $det(Q)$ and $det(Q_0)$, without loss of generality we shall assume $n=1$, $Q = (c)$.

Let $\{a_1, ..., a_d\}$ be a set of $\ZZ$ basis for $\mathcal{O}_K$. It would induce a basis of $\Cor_{K/\QQ}(V)$ in $\RR^d$ given by the base change matrix
$$ T = \left ( \begin{array} {cccc}
\sigma_1(a_1) & \sigma_1(a_2) & ... & \sigma_1(a_d) \\
\sigma_2(a_1) & \sigma_2(a_2) & ... & \sigma_2(a_d) \\
... & ... & ... & ... \\
\sigma_d(a_1) & \sigma_d(a_2) & ... & \sigma_d(a_d) \\
\end{array}
\right )
$$
And $Q_0 = T \cdot diag(\sigma_1(c), ..., \sigma_d(c)) \cdot T^t$. It is not hard to see that $T^tT$ gives the trace form of $K$. Therefore $det(Q_0) = det(T^tT) \cdot det (diag(\sigma_1(c), ..., \sigma_d(c))) = disc(K/\QQ)\cdot Norm_{K/\QQ}(c)$, as desired.
\end{proof}

In fact, it is always possible to find a copy of $\Cor_{K/\QQ}(B^0)$ inside $\Cor_{K/\QQ}(B)$ for any quaternion $B$. We simply observe that for any element $(\gamma^{\sigma_1}, \gamma^{\sigma_2}, ..., \gamma^{\sigma_d}) \in Cor_{K/\QQ}(B^0)$, we can map it to $\phi (\gamma)$ (cf. Section 2.2 of this paper).

It remains to be proven that $\Cor_{K/\QQ}(B^0, Q_B)$ is the $9$-dimensional Hodge substructure we wanted. We observe the following facts:

\begin{lemma}We use $V$ to denote the weight $1$ Hodge structure of Mumford fourfold in this lemma. Then $\Cor_{K/\mathbb{Q}}(B)$ is isomorphic to the weight two space $V \otimes V$. 
\end{lemma}

\begin{proof} This comes from the fact that representations of symplectic group are self-adjoint (cf. \cite{F-H91}). 

We have $\Cor_{K/\mathbb{Q}}(B) \cong M_{8}(\mathbb{Q}) \cong V \otimes V^*$ according to our hypothesis. We also know $G$ is a subgroup of $Sp_{8}(\mathbb{Q})$ since it preserves a symplectic form. Therefore, as a representation of $G$, $\Cor_{K/\mathbb{Q}}(B) \cong V \otimes V^* \cong V \otimes V$. 

Furthermore, by the classification of representations of $Sp_{2n}$ we see that $\Nm(G)$ acts on $\Cor_{K/\mathbb{Q}}(B)$ by conjugation. Since the Killing form $Q_B$ is the only invariant quadratic form on the adjoint representation of $B^0 = Lie(G)$, this shows that $\Cor_{K/\QQ}(B^0, Q_B)$ is the unique Hodge substructure in $Sym^2(H^1(A, \QQ))$ of K3 type endowed with an intersection form invariant under the Mumford-Tate group action defined over $\QQ$.

\end{proof}

\subsection{Discriminant of the quadratic form, saturating via maximal orders}

The basis in lemma 3.1 is not optimal, for the obvious reason that it depends heavily upon the generators chosen for the quaternion, which in most cases admits no canonical choice. In this section we shall introduce maximal orders of quaternion and some of their properties, and use them to derive a more canonical formula. It is not difficult to see that the construction introduced in this section, and the discriminant calculated, can be generalized to any quaternion $B$ over an arbitrary totally real number field $K$.

\begin{definition}
(Definition 10.2.1 in \cite{Voi17}) Let $B$ be a finite-dimensional $K$-algebra. Then an \textbf{$\mathcal{O}_K$-order} $O$ is a projective $\mathcal{O}_K$-lattice that is also a subring of $B$.

(Definition 10.4.1 in \cite{Voi17}) An $\mathcal{O}_K$-order is \textbf{maximal} if it is not properly contained in another $\mathcal{O}_K$-order. 
\end{definition}

We need the following classical properties of maximal orders:

\begin{proposition}(Local maximal property, see Lemma 10.4.2 in \cite{Voi17})Suppose $R$ is a Dedekind domain and $F$ the fraction field of $R$, and $B$ a quaternion over $F$, $O$ an $R$-order of $B$. Then $O$ is a maximal order if and only if $O_{(\mathfrak{p})}$ is a $R_{(\mathfrak{p})}$-maximal order for all primes $\mathfrak{p}$ of $R$.

(Corollary 10.3.3 in \cite{Voi17}) If $O$ is an $\mathcal{O}_K$-order, then every element $\gamma \in O$ is integral over $\mathcal{O}_K$, that is, $\gamma$ is annihilated by a minimal polynomial with coefficients in $\mathcal{O}_K$.

(Corollary 10.3.6 in \cite{Voi17}) Furthermore, if $B$ comes with a standard involution (in our situation this is always satisfied), then $\gamma \in B$ is integral over $\mathcal{O}_K$ if and only if $trd(\gamma), nrd(\gamma) \in \mathcal{O}_K$.
\end{proposition}

Therefore we can define a $\mathcal{O}_K$-integral form (also called the \textbf{trace form}) by the following:

$$B \times B \rightarrow K$$
$$(\gamma, \gamma') \mapsto trd(\gamma\gamma')$$

A straightforward computation shows that when restricted to the trace $0$ space $B^0 = span_K (\alpha, \beta, \alpha\beta)$, the above trace form is given by the matrix  
$$ \left ( \begin{array}{ccc}
2a &  &  \\
 & 2b &  \\
 &  & -2ab \end{array} \right) $$
which coincides with the Killing form.

With respect to this trace form we have the following definition:

\begin{definition}
(15.1.2 in \cite{Voi17}) 
The structure theorem of modules over Dedekind domain states that we can write an order $O$ into $$O = \mathfrak{a}_1e_1 \oplus \mathfrak{a}_2 e_2 \oplus \mathfrak{a}_3 e_3 \oplus \mathfrak{a}_4 e_4 $$
as $\mathcal{O}_K$ modules, where $\mathfrak{a}_i$'s are fractional ideals.
The \textbf{discriminant} of an order $O$ is given by $$disc(O) := \Pi_{i=1}^4 (\mathfrak{a}_i)^2 det|trd(e_i e_j)_{i,j = 1,...,4} |$$
where $|trd(e_i e_j)_{i,j = 1,...,4}|$ is a $4 \times 4$ matrix with $(i,j)$-th entry given by $trd(e_i e_j)$.
\end{definition}
\begin{remark}
Some will define the discriminant to be the square root of $disc(O)$.
\end{remark}

And we have the following fact:

\begin{fact}
(Theorem 23.2.9 in \cite{Voi17}) Let $O$ be an order of quaternion $B$ over a number field $K$ with discriminant $\mathfrak{D}$. Then $O$ is maximal if and only if $$disc(O) = \mathfrak{D}^2$$
\end{fact}

We would like to use this fact and saturate our lattice, since when we wrote down the basis for $V_{K3}$, we were intrincially choosing a basis of $B^0$ given by $\alpha, \beta, \alpha \beta$. There is no guarantee that this basis is saturated. The goal is to produce as large as possible an $\mathcal{O}_K$-integral lattice on $B^0$ with respect to the Killing form.

We consider the sublattice $O \cap B^0$, for a maximal order $O$ such that the order $\mathcal{O}_K \{ 1, \alpha, \beta, \alpha \beta\}$ is properly contained in $O$. By the definition of maximal orders and trace form, and the observation that trace form coincides with the Killing form on $B^0$, we see the sublattice we obtained in this way is indeed integral.

The goal now is to compute the discriminant of the Killing form on $O \cap B^0$. We start by assuming that $B$ does not ramify at places lying over $2$.

\begin{lemma}
For a quaternion $B$ not ramifying over $2$ and a maximal order $O$, we can always write $O$ into $\mathcal{O}_K \oplus M$, where $M$ is a projective $\mathcal{O}_K$-module of rank 3 satisfying $trd(\gamma) \neq 2\gamma$ for any $\gamma \in M$. Furthermore, there exists an element $\gamma_0$ in $M$ such that its reduced trace is $1$.
\end{lemma}

\begin{proof}The first part of the statement follows from the structure theorem of Dedekind domain (Theorem 7.3 in \cite{May09}). Another way of viewing the direct summand $\mathcal{O}_K$ is seeing it as the integral model of the scalars in $B$.

Since $B$ is not ramifying over $2$, $disc(O)$ is not divisible by $2$. If every element in $M$'s trace lies in $2\mathcal{O}_K$, then $disc(O)$ should be divisible by $2$, which contradicts our assumption.
\end{proof}

\begin{corollary}
$O$ can be written into $\mathcal{O}_K \oplus \mathcal{O}_K \gamma_0 \oplus M'$, where $M'$ consists of elements lying inside $B^0$
\end{corollary}

\textbf{Proof} The fact that $\gamma_0$ has minimal reduced trace suggests that $M' = M/\mathcal{O}_K \gamma_0$ as a module is torsion free. Moreover, we can always subtract a suitable copy of $\gamma_0$ from the generators $M'$ so that they become trace zero.

This leads to the following proposition:

\begin{proposition} Assume that $B$ does not ramify at places lying over $2$, we have
$$disc(O \cap B^0) = 2\mathfrak{D}^2$$
\end{proposition}

\begin{proof} The previous corollary implies $O \cap B^0 = \mathcal{O}_K (2\gamma_0 - 1) \oplus M'$. Therefore the lattice $\mathcal{O}_K \oplus (O \cap B^0)$ is an index 2 sublattice of $O$. With respect to the trace form, $\mathcal{O}_K \oplus (O \cap B^0)$ has discriminant $4disc(B)^2$.

Notice that the trace form admits an orthogonal decomposition on the summands of $\mathcal{O}_K \oplus (O \cap B^0)$, and $trd(1 \cdot 1) =2$. This implies the discriminant of the trace form, which coincides with the Killing form, on $O \cap B^0$ should be $2\mathfrak{D}^2$.
\end{proof}

\begin{corollary}Suppose $\mathfrak{p}$ is a non-dyadic place of $K$. Let $B$ be a quaternion over $K$, and $O$ a maximal order of $B$. Then the discriminant of $O_{(\mathfrak{p})} \cap B_{(\mathfrak{p})} ^0$ is $2\mathfrak{p}^2 $ if and only if $B$ ramifies at the place $\mathfrak{p}$.
\end{corollary}
\begin{proof}Directly from the local-global principle and local maximal property. 
\end{proof}

To extend the result to a dyadic place $v$, we recall that any quaternion over $K$ ramifying at $v$ is $K_v$-isomorphic to $$B_v = span_{K_v}\{1, \alpha, \beta, \alpha\beta: \alpha^2-\alpha +a = 0, \beta^2 = \pi, \alpha\beta = -\beta\alpha\}$$ where $F = K_v[x]/(x^2-x+a)$ is the unique unramified quadratic extension of $K_v$ and $\pi$ the uniformizer of $K_v$. We recall furthermore that the unique maximal order is given by $O_v = \mathcal{O}_F \oplus \beta \mathcal{O}_F$, viewed as an $\mathcal{O}_{K_v}$ algebra (cf. \cite{Voi17} Section 13.1). The reduced trace $trd$ on $O_v$ annihilates the $\mathcal{O}_{K_v}$ submodule $\beta\mathcal{O}_F$, and coincides with $tr_{F/K_v}$ on $\mathcal{O}_F$, and it can be linearly extended to $B_v$, hence it makes sense to talk about the trace zero $K_v$-vector space $B_v^0$, and the reduced trace form on $O_v$. 

\begin{lemma}
Let $B_v$, $\pi$, $F$, $O$, $trd$ be stated as above. Then with respect to the trace form on $O_v$, $disc(O_v \cap B_v^0) = 2 \pi^2$.
\end{lemma}
\begin{proof}
To begin with, $disc(O_v) = \pi^2$ with respect to the trace form. We notice, again, that there is a distinct element in $O_v$ that has reduced trace 1, namely $\alpha$ (by our definition). Therefore $O_v \cap B_v^0 \oplus O_{K_v} = O_{K_v}  \oplus O_{K_v}(2\alpha -1) \oplus \beta\mathcal{O}_F $ as an $\mathcal{O}_{K_v}$ module, which is an index $2$ sublattice of $O_v$, hence having discriminant $4\pi^2$. Again, with respect to the trace form, $O_{K_v}$ is orthogonal to $O_v \cap B_v^0$, and the discriminant of trace form on this submodule is $2$. Thus we conclude that $disc(O_v \cap B_v^0) = 2 \pi^2$. 
\end{proof}

Hence by the local-global principle and local maximal property of $O$, and the calculation we have done on the relation of the discriminant of a quadratic form and its corestriction (Lemma 3.2 of this paper), we have proven the following theorem:
\begin{theorem}For any quaternion $B$, $$disc(O \cap B^0) = 2\mathfrak{D}^2$$
Denote $\Lambda_{can}$ to be the lattice in $V_{K3}$ such that $\Lambda_{can}= \Cor_{K/\QQ}(O \cap B^0)$. Then $$disc(\Lambda_{can}) = 2^3 disc(K/\mathbb{Q})^3 \dot Norm_{K/\mathbb{Q}}(\mathfrak{D}^2)$$
\end{theorem}

\subsection{Twisting the intersection form, and an explicit example}

In this section we introduce a twist of the Killing form $Q$ that is also preserved by $G$. To begin with, we can perform a rescaling on $B^0$: $\alpha \mapsto b\alpha, \beta \mapsto a\beta, \alpha\beta \mapsto \alpha\beta$ and see that it transforms $Q_B$ into 
$$ \left ( \begin{array}{ccc}
2ab^2 &  &  \\
 & 2a^2b &  \\
 &  & -2ab \end{array} \right) $$
Taking away the common factors $ab$ in the diagonal we obtain a new quadratic form that is also preserved by $G$:
$$Q_{twi}:= \left ( \begin{array}{ccc}
2b &  &  \\
 & 2a &  \\
 &  & -2 \end{array} \right) $$
Although the canonical lattice will not be $\mathcal{O}_K$-integral with respect to $Q_{twi}$, and it will not be invariant with respect to the choice of generators anymore, the discriminant it generates still carries interesting arithmetic information from the original quaternion.

As an example, we shall calculate the canonical intersection form of the quaternion $B$ with a maximal order $O$ given in \cite{Voi09}, which is defined over the totally real cubic subfield of $\QQ(\zeta_9)$ with $disc(B)= 1$, i.e. it only ramifies at two infinite places. We will also compute $disc(\Cor_{K/\QQ}(O \cap B^0))$ with respect to $Q_{twi}$, which will become relevant in the last part of the paper.
\begin{example}
Consider the quaternion $B$ defined over the totally real cubic subfield $K$ of $\QQ(\zeta_9)$ ramifying only at two infinite places. \cite{Voi09} has given out the explicit generators for $B$ and a maximal order $O$:
$$b := -(\zeta_9+\overline{\zeta_9}), B =\left( \frac{-3,b}{K} \right)$$
$$
O := \mathcal{O}_K \oplus \mathcal{O}_K \zeta \oplus \mathcal{O}_K \eta \oplus \mathcal{O}_K \omega, \text{with }
\zeta = -\frac{1}{2}b + \frac{1}{6}(2b^2-b-4)\alpha, $$$$
\eta = -\frac{1}{2}b\beta + \frac{1}{6}(2b^2-b-4)\alpha\beta, 
\omega = -b + \frac{1}{3}(b^2-1)\alpha -b\beta + \frac{1}{3}(b^2-1)\alpha\beta 
$$

Clearly, $O \cap B^0$ is an $\mathcal{O}_K$-lattice generated by: $\zeta' =  \frac{1}{3}(2b^2-b-4)\alpha$, 
$\eta = -\frac{1}{2}b\beta + \frac{1}{6}(2b^2-b-4)\alpha\beta$, and $\omega' =  \frac{1}{3}(b^2-1)\alpha -b\beta + \frac{1}{3}(b^2-1)\alpha\beta$. The Killing form $Q_B$ with respect to the basis $\{\zeta', \eta, \omega'\}$ is 
$$ Q_B = \left ( \begin{array}{ccc}
2b^2-8 & 0 & -2 \\
0 & 2b & 4b+1 \\
-2 & 4b+1 & 8b+2 \end{array} \right) $$
Since $\mathcal{O}_K = \ZZ(b)$, we can pick a basis of $\Lambda_{can}$ to be $\{\phi(\zeta'), \phi(b\zeta'), \phi(b^2\zeta'), \phi(\eta), \phi(b\eta), ... \phi(b^2\omega') \}$, where the $\phi$ map is introduced in Section 2.2 of this paper. Then with respect to this basis, the intersection form is given by $$\left ( \begin{array}{ccccccccc}
-12 & 6 & -12 & 0 & 0 & 0& -6 & 0 & -12 \\
6 & -12 & 6 & 0 & 0 & 0 & 0 & -12 & -6 \\
-12 & 6 & -30 & 0 & 0& 0& -12 & -6 & -36 \\
0 & 0 & 0 & 0 & 12 & 6& 3 & 24 & 18 \\
0 & 0 & 0 & 12 & 6 & 36& 24 & 18 & 75 \\
0 & 0 & 0 & 6 & 36 & 30& 18 & 75 & 78 \\
-6 & 0 & -12 & 3 & 24 & 18 & 6 & 48 & 36 \\
0 & -12 & -6 & 24 & 18 & 75 & 48 & 36 & 150 \\
-12 & -6 & -36 & 18 & 75 & 78 & 36 & 150 & 156 \\
  
\end{array}\right )$$
The discriminant form of $\Lambda_{can}$ is $(\ZZ/2\ZZ)^3 \times (\ZZ/3\ZZ)^6 \times (\ZZ/9\ZZ)^3$. \footnote{At this point it would be tempting to list out all possible 11-dimensional positive definite lattices with the same discriminant, which will give an enumeration of all possible elliptic fibration structures in the Picard group of some K3 surfaces with their transcendental lattice given by this $9\times 9$ matrix. It turns out that this computation crashes Magma. Even with a quaternion defined over the totally real subfield of $\QQ(\zeta_7)$ ramifying at two infinite places, the count of possible elliptic fibrations is still too high for Magma to process.}
%%On the other hand, $Q_{twi}$ with respect to the basis $\{\zeta', \eta, \omega'\}$ is given by the matrix
%%$$ \left ( \begin{array}{ccc}
%%\frac{2}{3}(b-1) & 0 & \frac{2}{3}b \\
%%0 & \frac{1}{3}(-4b^2-2) & -3b^2-\frac{1}{3} \\
%%\frac{2}{3}b & -3b^2-\frac{1}{3} & -6b^2+\frac{2}{3} \end{array} \right) $$
%%The discriminant of this matrix is $\frac{2}{3}b$. Therefore $$disc(\Cor_{K/\QQ}(O \cap B^0), Q_{twi}) = disc(K/\QQ)^3 Norm_{K/\QQ}(\frac{2}{3}b) = 2^3 \cdot 3^9$$
\end{example}

\section{Two applications}
We will utilize the invariant we produced in two ways. First of all, combining the results from lattice theory, we shall show that all the K3 surfaces arising from this construction are elliptically fibered and the torsion part of their Mordell-Weil group is always trivial. Another application will be using the invariant to locate a special Mumford fourfold that is proven to be derived equivalent to Shioda's fourfold in \cite{myself}.

\subsection{General facts about elliptic surfaces}

We review basic properties of elliptic fibration $f: S \rightarrow C$ with a section. The generic fiber is denoted as $E$, an elliptic curve. The zero section is denoted $O$. We will assume, furthermore, that $S$ has singular fibers (in particular, it is not isomorphic to $E \times C$). We have the following fact:

\begin{fact}
(Theorem 6.12 in \cite{S-S09}) For an elliptic surface $S$ with a section $C$, $Pic^0(S) \cong Pic^0(C)$.
\end{fact}
\begin{corollary}
If $f: S \rightarrow \mathbb{P}^1$ with a section, then $Pic(S) \cong NS(S)$.
\end{corollary}

The Kodaira classification theorem states the following:

\begin{fact}
(4.1 in \cite{S-S09}) If a singular fiber is irreducible, then it is a rational curve with a node or a cusp. If a singular fiber is reducible, then each irreducible component is a smooth rational curve with self-intersection (-2) on the elliptic surface. Furthermore, the dual graph of each fiber is a Dynkin diagram of affine $ADE$ type.
\end{fact}

The $NS(S)$ has contributions coming from two origins: the contribution of irreducible components of singular fibers, zero sections and generic fibers are the ones called \textbf{vertical components}, and the ones coming from the Mordell-Weil group of $E$ (each point $P$ on $E$ determines a section $\overline{P}$ on $S$) are called \textbf{horizontal components}. To begin with, we have the following fact:

\begin{fact}
(Theorem 6.3 in \cite{S-S09}) Let $T$ denote the subgroup of $NS(S)$ generated by the zero section and fiber components. Then the map $P \mapsto \overline{P}$ \emph{mod} $T$ gives an isomorphism $$E(K) \cong NS(S)/T.$$
\end{fact}

We can say a little bit more about $T$, which is often referred to as \textbf{trivial lattice}. (Section 6.4 in \cite{S-S09}) For an elliptic surface $f: S \rightarrow C$ with a zero section $\overline{O}$, we denote $F$ to be a general fiber, $R$ the points of $C$ underneath reducible fibers, $F_v$ the fiber $f^{-1}(v)$ above $v \in C$, $\Theta_{v, 0}$: the component of $F_v$ met by the zero section and $\Theta_{v, i}$ the other components of $F_v$ ($i = 1, ..., m_v -1$) where $m_v$ is the number of components of the fiber $F_v$. We denote $T_v$ as the lattice generated by fiber components of $F_v$ not meeting the zero section: $T_v = \langle \Theta_{v, i}; i = 1, ..., m_v -1 \rangle$.

Using the above notation, the trivial lattice is defined as the orthogonal sum $$T = \langle \overline{O}, F \rangle \oplus \bigoplus_{v \in R} T_v $$
And we have the following proposition:
\begin{proposition}
(Proposition 6.6 and Section 6.5 in \cite{S-S09}) The divisor classes of $\{ \overline{O}, F, \Theta_{v, i}; v \in R, i = 1, ..., m_v -1  \}$ form a $\mathbb{Z}$-basis of $T$. The $ADE$ type of a reducible fiber is recoverable by omitting the component intersecting zero section.

(Lemma 8.3 in \cite{S-S09}) Assuming that the Euler characteristic $\chi(S) > 1$, then $\bigoplus_{v \in R} T_v$ is exactly given by the direct sum of $ADE$-type lattices.
\end{proposition}

To complete the picture, we need to describe the fibers corresponding to the torsion points on $E(K)$ and the horizontal components coming from torsion-free part of $E(K)$:

\begin{proposition}
(Section 7.1 in \cite{S-S09}) The \textbf{primitive closure} $T'$ of $T$ in $NS(S)$ is defined to be $$T' = (T \otimes \mathbb{Q}) \cap NS(S)$$ 
(Corollary 7.2 in \cite{S-S09}) $E(K)_{tors} \cong T'/T$
\end{proposition}

In particular, if $E(K)_{tors}$ is trivial, then $T$ embeds primitively into $NS(S)$. (Section 11.1 in \cite{S-S09}) But in general this is not the case. We define the \textbf{essential lattice} to be $$L(S) = T(S)^{\bot} \subset NS(S)$$
And lattice theory gives

$$| disc(L(S))| = \big | \frac{disc(NS(S))}{disc(T(S))} \big | |NS(S):(L(S) \oplus T(S)) |^2$$

There is also a map (Lemma 11.2, Section 11.4 in \cite{S-S09}) from $E(K)$ to $L(S)_{\mathbb{Q}}$ such that a point $P$ is being taken to a representative congruent to $\overline{P} = f^{-1}(P)$ modulo $T(S)_{\mathbb{Q}}$.

Moreover, there is a proposition by Nikulin that states the characteristic of $NS(S)$ as an overlattice of $L(S) \oplus T(S)$.

\begin{proposition}
(Proposition 1.4.1 in \cite{Nik80}) Let $\Lambda$ be an even lattice. A given embedding $\Lambda \hookrightarrow \Lambda'$ of even lattices, such that $\Lambda'/\Lambda$ is a finite abelian group, will be called an \textbf{even overlattice} of $\Lambda$. If we denote $H_{\Lambda'} = \Lambda'/\Lambda$, then since we have the chain of embeddings $\Lambda \hookrightarrow \Lambda'  \hookrightarrow \Lambda^{'*} \hookrightarrow \Lambda^*$, $H_{\Lambda'}$ is obviously a subgroup of $A_{\Lambda} = \Lambda^* / \Lambda$. Moreover, $(\Lambda^{'*}/\Lambda)/H_{\Lambda'} = A_{\Lambda'}$

This gives rise to a correspondence $\Lambda' \rightarrow H_{\Lambda'}$. In fact it determines a bijection between even overlattices of $\Lambda$ and subgroups $H$ of $A_{\Lambda}$ such that $q_{\Lambda}|H = 0$. In particular, $(q_{\Lambda}|(\Lambda^{'*}/\Lambda))/H_{\Lambda'} = q_{\Lambda'}$
\end{proposition}

\subsection{First application: Elliptic fibration structure of the K3 surfaces determined by $V_{K3}$}

We start with a lemma that guarantees all K3 surfaces we are considering admits an elliptic fibration.

\begin{lemma}
All K3 surfaces $S$ obtained via Mumford-Galluzzi constructions admit elliptic fibrations with section $f: S \rightarrow \mathbb{P}^1$.
\end{lemma}

\begin{proof} A result by Kondo \cite{Kon92} states that a K3 surface admits an elliptic fibration with a section to $\mathbb{P}^1$ if its Picard group is the direct sum of $U$ with a negative-definite even lattice. Another result by Nikulin (cf. Corollary 1.4.3 in \cite{Dol82}) states that an even indefinite lattice $\Lambda$ with $rank \Lambda \geq \lambda(\Lambda) +3$, there exists some even lattice $\Lambda'$ such that $\Lambda \cong \Lambda' \oplus U$. This condition is satisfied in our case, since $Pic (S)$ has signature $(1, 12)$ and $13= rank (Pic(S)) > 9+3 \geq \lambda(Pic(S)) +3$. Done.
\end{proof}

This does not imply that the elliptic fibration structure is unique. In most cases, the K3 surface can have multiple different elliptic fibration structures.

\begin{example}
Consider the example given in Section 3.3 of this paper. The $ADE$ lattice of the resulting Picard group can be the following: $A_1$, $A_1^{\oplus 2}$, ..., $A_1^{\oplus 4}$, $A_1\oplus A_2, A_1^{\oplus 2}\oplus A_2$, $A_1^{\oplus 3}\oplus A_2, A_1 \oplus A_2^{\oplus 2}$, $D_4$.
\end{example}

It is not hard to see, simply by counting the number of generators, that the rank of $ADE$ singular fibers is always not greater than $10$ for any K3 surface arising from Mumford's construction. And we have the following theorem:

\begin{theorem}
All Mumford-Galluzzi K3 surface with the canonical intersection form has trivial Mordell-Weil torsion group, except when the canonical intersection form can be realized as the discriminant form of a rank $3$ lattice. In that case, the Mordell-Weil torsion group can be $\ZZ/2\ZZ$.
\end{theorem}

\begin{proof} We look at the list of $ADE$ fibers with rank not greater than $10$ in \cite{Shi05} (cf. Theorem 2.14 in \cite{Shi05}). Only in three cases, namely $\bigoplus_{v \in R} T_v = 8A_1, 9A_1$ or $A_3 + 6A_1$, can the Mordell-Weil torsion group be nontrivial, which is $\mathbb{Z}/2\mathbb{Z}$. Notice that in any case, the discriminant of $\bigoplus_{v \in R} T_v $ only has $2$-torsions.

Theorem 7.1 in \cite{Shi05} states the if and only if condition of a lattice $\Sigma$ of $ADE$-type and a finite abelian group $G$ to be realized as  $ADE$ fibers of an elliptic K3 surface with $G$ isomorphic to the torsion of Mordell-Weil group. It is noticable that in all cases the overlattice of $ADE$ lattice only admits $2$-torsions or $4$-torsions. Since there are no totally real cubic field with discriminant a power of $2$, the essential lattice of the elliptic fibration will have rank at least $3$. Therefore, the only configuration when the Mordell-Weil group can have nontrivial torsion group is when the $ADE$ lattice is $8A_1$, and the discriminant form of the original lattice is length $3$ (meaning that it can be realized as a discriminant form of a lattice of rank at least $3$).

\end{proof}

\subsection{General facts about Shioda's abelian fourfold}

When the Mumford's fourfold has complex multiplication, general theory of CM abelian varieties states that the Mumford-Tate group in this case degenerates into a unitary group defined over $K$. It was proven in \cite{myself} that we can obtain a CM Mumford's fourfold, denoted $A_M$, by modifying the Hodge structure of the Jacobian $A_S$ of $y^2 = x^9-1$. Furthermore, we know that the ground field of the quaternion generating this example is given by the totally real subfield of $\QQ(\zeta_9)$, which we will denote as $K$ in this section. We also know that the CM field of $A_M$ is $\QQ(\zeta_3)$. We will show that $A_M$ is constructed by the quaternion ramifying only at two infinite places by computing the discriminant of its $V_{K3}$.

First we recall Deligne's construction of CM abelian varieties (see \cite{Del82} Example 3.7).  

By CM-field we mean a quadratic totally imaginary extension of a totally real field; a CM-algebra is a finite product of CM-fields. Let $E = E_1 \times ... \times E_n$ be such an algebra. Then there exists an involution $\iota$ acting by complex conjugation on each of the factors of $E$. Let $F = F_1 \times ... \times F_n$ denote the subalgebra in $E$ fixed by $\iota$.

We denote $S$ for the set $Hom_{\QQ}(E, \CC)$. Then a CM-type for $E$ is a subset $\Sigma \subset S$ such that $S = \Sigma \sqcup \iota\Sigma$.

The complex structure is given by the following decomposition:
$$E \otimes_{\QQ} \CC \cong \CC^S = \CC^{\Sigma} \oplus \CC^{\iota\Sigma}.$$

Let $\CC^{\Sigma}$ be the $V^{1,0}$ space and $\CC^{\iota\Sigma}$ its complex conjugate. Then we can view $E$ as $H_1(A, \QQ)$ of some CM abelian variety $A$, for example, given by $A(\CC) = \CC^{\Sigma} / \Sigma(\mathcal{O}_E)$. One can see that the $H_1$ with $\QQ$-coefficients of any CM abelian varieties arises from this construction.

With this language, we can describe $A_S$ via the Artinian ring $E := \QQ[x]/(x^8 + x^7 + ... +x +1) \cong \mathbb{Q}(\zeta_9) \times \mathbb{Q}(\zeta_3) = H_1(A, \mathbb{Q})$; $H_1(A, \mathbb{Z})$ can be obtained by considering the embedding of products of fractional ideals in $\mathbb{Z}(\zeta_9)$ and $\mathbb{Z}(\zeta_3)$ into $H_1(A, \mathbb{Q})$ up to a positivity condition which is induced by the Riemann condition. Therefore $A_S$ is isogeneous to the abelian variety obtained by $E/ \mathcal{O}_E$ with the complex structure given by $ E_{\CC} = \CC^{\Sigma} \oplus \CC^{\overline{\Sigma}} \oplus \CC^{\tau} \oplus \CC^{\overline{\tau}} $
where $\Sigma$ denotes the set of embeddings $\QQ(\zeta_9) \hookrightarrow \CC$ such that their restriction onto $\QQ(\zeta_3)$ is identical. In particular, $A_S$ splits into a product of abelian threefold $A_0$ and elliptic curve.

Furthermore we recall the Mumford-Tate representation of $A_S$ and $A_M$ respectively:

\begin{proposition}
The Mumford-Tate group of $A_S$ is $$G_S = \{ x \in \mathbb{Q}(\zeta_9) | Nm_{\mathbb{Q}(\zeta_9)/\mathbb{Q}(\zeta_9 + \overline{\zeta_9})}(x) = 1 \}$$
The representation of $G_S, \mathbb{R}$ on $E_\mathbb{R}$ is given by 
\begin{center}
\begin{tabular}{ c c c }
 $G_S, \mathbb{R} \cong U(1)^3 \hookrightarrow$ & $SU(2)^3 \hookrightarrow$ & $SU(2)^4$ \\ 
   & $(u,v,w) \mapsto$ & $(u,v,w,uvw)$ \\  

\end{tabular}
\end{center}

The Mumford-Tate group $G_M$ of $A_M$ is isomorphic to $G_S$ over $\RR$. Up to an automorphism on $G_M$ the Mumford-Tate representation of $G_M, \RR$ is

\begin{center}
\begin{tabular}{ c c c }
 $G_M, \mathbb{R} \cong U(1)^3 \hookrightarrow$ & $SU(2)^3 \hookrightarrow$ & $SU(2)^4$ \\ 
   & $(u,v,w) \mapsto$ & $(u^{-1},v,w,uvw)$ \\  

\end{tabular}
\end{center}

\end{proposition}

\subsection{Second application: the CM Mumford's fourfold that is derived equivalent to Shioda's fourfold}

By observing the Mumford-Tate representation, we can say that switching two of the complex conjugate embeddings $\Sigma = \{\sigma_1, \sigma_2, \sigma_3 \}$ in $  E_{\CC} = \CC^{\Sigma} \oplus \CC^{\overline{\Sigma}} \oplus \CC^{\tau} \oplus \CC^{\overline{\tau}} $ will make $A_S$ into an $A_M$. Since the discriminant we calculated is a $(2,2)$ form, it is invariant under this perturbation. It is then possible to compute $disc(\Lambda)$ by finding the right analogue on the $A_S$ side, where $\Lambda$ denotes the lattice in $V_{K3}$ induced by $H^1(A_M, \ZZ)$, which can have the canonical intersection form or the twisted intersection form. We do so by carefully tracking the change of basis we are taking:
$$\left ( \begin{array}{c}
\text{Invariants and determinants} \\
\text{in} \ H^1(A_S) \ \text{and} \ H^2(A_S)
\end{array} \right )
\xrightarrow[\text{change}]{\text{coordinate}}
\left ( \begin{array}{c}
\text{Invariants and determinants} \\
\text{of classical }SU(2,\RR)^3\ \text{reps}
\end{array} \right )
$$

To start with, we know $H^1(A_S, \ZZ)$ is the direct sum of two fractional ideals of $\QQ(\zeta_9)$ and $\QQ(\zeta_3)$. Since the class group of both field is trivial, $A_S$ is isogenous to $E/\ZZ(\zeta_9) \oplus \ZZ(\zeta_3)$. 

Next we show that the elliptic curve component does not enter the picture:

\begin{lemma}
The analogue space of $V_{K3}$ in $Sym^2(H^1(A_S, \QQ))$ lies inside $Sym^2(H^1(A_0, \QQ))$.
\end{lemma}
\begin{proof}
The weight $(1,0)$ vector of the elliptic curve part has $SU(2)^3$ representation weight $(1,1,1)$, which will generate the $27$-dimensional space when we take symmetric power. 
\end{proof}

In fact we can do better than this: we know $H^1(A_0, \CC) = W_1 \oplus W_2 \oplus W_3$ where each $W_i$ is a complex conjugation pair of embedding $\QQ(\zeta_9) \hookrightarrow \CC$. The distinguished quadratic form on $W_i$ that the Mumford-Tate group preserve will be the trace form $\langle x,y \rangle = Tr_{\QQ(\zeta 9)/K}(xy) $. We can extend this form to $Sym^2(W_i)$, but we need to choose the right multiple of it.

\begin{lemma}
The standard $SU(2)$ preserves the norm form of $\CC$ over $\RR$. The discriminant of trace form is $4$ with respect to $\{1,i\}$ basis.

The corresponding discriminant of the Killing form is $2\cdot 4^3$. 
\end{lemma}

Therefore the discriminant of the induced quadratic form on $Sym^2(W_i)$ will have discriminant $2\cdot 3^3$. We see that upon taking corestriction, $disc(\Lambda) = 2^3 \cdot 3^9 \cdot 3^{12}$. Notice that $disc(\Lambda)$ has only two prime factors 2 and 3, we see that the only possible local places of ramification for $B$ are those that lies above $2$ and $3$. But since $2$ is still a prime in $\mathcal{O}_K$, and $3$ completely ramified, $B$ can ramify at neither of the place while satisfying the corestriction trivial condition. Hence $B$ does not ramify at any finite place.

%%However, one might notice that $disc(\Lambda)$ cannot be realized as a sublattice of $\Lambda_{can}$ because $3^9$ is not a square. This can be resolved by looking at the twisted intersection form $Q_{twi}$ on $\Lambda_{can}$. By our previous calculation, $\Lambda$ can be realized as a sublattice of $\Lambda_{can}$ of index $3^6$ (enough to make $Q_{twi}$ integral with respect to $\Lambda$) with the quadratic form given by the twisted intersection form. Hence we have proved the following:

\begin{theorem}
$A_M$ is a CM point on the Shimura curve obtained by the quaternion $B$ over the totally real subfield of $\QQ(\zeta_9)$ that only ramifies at two infinite places.
\end{theorem}


\begin{thebibliography}{}



\bibitem{Mum69}
  Mumford, David,
  \emph{A Note of Shimura's Paper "Discontinuous Groups and Abelian Varieties"},
  Math. Ann. 181, 345 - 351,
  1969

\bibitem{Nik80}
  Nikulin, Vyacheslav Valentinovich,
  \emph{Integral symmetric bilinear forms and some of their applications},
  Izvestiya: Mathematics 14.1: 103-167,
  1980

\bibitem{Shi81}
Shioda, Tetsuji. "Algebraic cycles on abelian varieties of Fermat type." Mathematische Annalen 258.1 (1981): 65-80.

\bibitem{Del82}
Deligne, Pierre. "Hodge cycles on abelian varieties." Hodge cycles, motives, and Shimura varieties. Springer, Berlin, Heidelberg, 1982. 9-100.

\bibitem{Dol82}
 Dolgachev, Igor,
 \emph{Integral quadratic forms: applications to algebraic geometry},
 Séminaire Bourbaki 25. No. 611, 251-278,
 1982-1983.

\bibitem{F-H91}
  Fulton, William, and Joe Harris.
  \emph{Representation Theory: A First Course},
  GTM 129 Springer-Verlag,
  1991

\bibitem{Kon92}
  Kondo, Shigeyuki,
  \emph{Automorphisms of algebraic K3 susfaces which act trivially on Picard groups},
  Journal of the Mathematical Society of Japan 44.1: 75-98,
  1992

\bibitem{Gal99}
Galluzzi, Federica. \emph{Corestriction of central simple algebras and families of Mumford-type}
Atti della Accademia Nazionale dei Lincei. Classe di Scienze Fisiche, Matematiche e Naturali. Rendiconti Lincei. Matematica e Applicazioni 10.3 
(1999): 191-211.

\bibitem{Noo00}
 Noot, Rutger. \emph{Abelian varieties with l-adic Galois representation of Mumford's type}, 
Journal fur die Reine und Angewandte Mathematik 519 (2000): 155-170.

\bibitem{Gal00}
  Galluzzi, Federica,
  \emph{Abelian Fourfold of Mumford-type and Kuga-Satake Varieties},
  Indag. Mathem., N.S., 11 (4), 547-560, 
  December 18, 2000

\bibitem{Shi05}
  Shimada, Ichiro,
  \emph{On elliptic K3 surfaces},
  arXiv preprint math/0505140,
  2005

\bibitem{May09}
 May, J. P. 
 \emph{Notes on Dedekind rings}, 
 2009 1-11.

\bibitem{S-S09}
  Schütt, Matthias, and Tetsuji Shioda,
  \emph{Elliptic surfaces},
   arXiv preprint arXiv:0907.0298,
  2009
  
\bibitem{Voi09}
 Voight, John. \emph{Shimura curve computations},
 Arithmetic geometry 8 (2009): 103-113.

\bibitem{Ser13}
 Serre, Jean-Pierre,
 \emph{Galois cohomology}
 Springer Science and Business Media, 
 2013.

\bibitem{Voi17}
  Voight, John, 
  \emph{Quaternion Algebras},
  Preprint, accessable via \url{https://math.dartmouth.edu/~jvoight/quat.html},
  2017

\bibitem{myself}
 Zhu, Yuwei, \emph{Constructing a CM Mumford fourfold from Shioda's fourfold}, arXiv: 1810.10058, 2018 


\end{thebibliography}
\end{document}